\documentclass[12pt, a4paper]{amsart} 
\usepackage[T2A]{fontenc}
\usepackage[utf8]{inputenc}
\usepackage{tikz}
\usepackage{geometry}
\usepackage{amsmath, amssymb, amsfonts, amsthm}
\usepackage{enumitem}
\usepackage{graphicx}
\usepackage{tikz}
\usepackage{hyperref}
\usepackage{bbm}

\usepackage{comment}
\usepackage[utf8]{inputenc}
\usepackage[T2A]{fontenc}
\usepackage[english]{babel}

\newtheorem*{ack}{Acknowledgements}
\newtheorem{theorem}{Theorem}
\newtheorem{theom}{Theorem}

\newtheorem{lemma}{Lemma}
\newtheorem*{prop}{Proposition}
\newtheorem{defin}{Definition}
\newtheorem{rem}{Remark}

\newenvironment{proof}{\noindent{\bf Proof:}}{$\hfill \Box$ \vspace{10pt}}  
\def\Xint#1{\mathchoice
{\XXint\displaystyle\textstyle{#1}}%
{\XXint\textstyle\scriptstyle{#1}}%
{\XXint\scriptstyle\scriptscriptstyle{#1}}%
{\XXint\scriptscriptstyle\scriptscriptstyle{#1}}%
\!\int}
\def\XXint#1#2#3{{\setbox0=\hbox{$#1{#2#3}{\int}$ }
\vcenter{\hbox{$#2#3$ }}\kern-.6\wd0}}

\def\dashint{\Xint-}
\newcommand{\hel} {
\hskip2.5pt{\vrule height7pt width.5pt depth0pt}
\hskip-.2pt\vbox{\hrule height.5pt width7pt depth0pt}
\, }
\newcommand{\restr}{\hel}

\begin{document}

\title{On the multidimensional Nazarov lemma.}
\author{Ioann Vasilyev}
\thanks{The author was supported by the Russian Science Foundation (grant No.~18-11-00053).}
\address{Universit\'e Paris-Est, LAMA (UMR 8050), 61 avenue du G\'en\'eral de Gaulle, 94010, Cr\'eteil, France}
\address{St.-Petersburg Department of V.A. Steklov Mathematical Institute, Russian Academy of Sciences (PDMI RAS), Fontanka 27, St.-Petersburg, 191023, Russia}
\email{milavas@mail.ru}

\begin{abstract}
In this article we prove a multidimensional version of the Nazarov lemma. The proof is based on an appropriate generalisation of the regularised system of intervals introduced in~\cite{nazhav} to several dimensions.  
\end{abstract}

\maketitle
\section{Introduction}
In the paper~\cite{nazhav} by V. Havin, F. Nazarov and J. Mashreghi the authors prove the following result.
\begin{theom}
\label{thm0}
Let $\Omega\in L^1(\mathbb R, dx/(1+x^2))\cap \mathrm{Lip}(\mathbb R)$ be a positive function. Then for each $\varepsilon > 0$ there exists a function $\Omega_1$, satisfying 
\begin{enumerate}
\item $\Omega(x) \leq \Omega_1(x)$ for all $x\in \mathbb R;$
\item $\Omega_1 \in L^1(\mathbb R, dx/(1+x^2));$
\item $\mathcal H\Omega_1 \in \mathrm{Lip}(\varepsilon, \mathbb R),$ where $\mathcal H$ is the Hilbert transform on the real line.
\end{enumerate}
\end{theom}

This theorem, sometimes called Nazarov's lemma was crucial in~\cite{nazhav} for the beautiful ``geodesic'' proof of the so-called First Beurling--Malliavin theorem. The latter theorem is among the deepest results in harmonic analysis. To illustrate this, we mention the articles~\cite{nazhav},\cite{polt},\cite{makpol},\cite{bermal1},\cite{bermal2},\cite{belbarul}, where Theorem~\ref{thm0} was shown to apply many results in the theory of exponential systems. We especially emphasize that Theorem~\ref{thm0} is a cornerstone of the proof of the famous Second Beurling--Malliavin Theorem. On top of that, Theorem~\ref{thm0} was recently used by J. Bourgain and S. Dyatlov in the theory of resonances for hyperbolic surfaces, see~\cite{bourdyat}. 

This article is devoted to the multidimensional version of Nazarov's lemma. Before stating our main result, let us first recall some definitions and fix some notations.

\begin{defin}
The \textbf{multidimensional Poisson measure} $dP_n$ is defined by the following formula $$dP_n(x):=\frac{dx}{(1+|x|^2)^{\frac{n+1}{2}}}.$$
If $n=2$, we shall omit the subscript $2$ and simply write $dP$ instead of $dP_2.$
The corresponding weighted Lebesgue space $L^1(dP)$ is the space of all functions $f$ satisfying $\int_{\mathbb R^2}fdP<\infty.$
\end{defin}
\begin{defin}
The \textbf{Riesz transformations} of a function $f\in L^{1}(dP)$ for $j=1,2$ are defined as the following principal value integrals
$$\mathcal R_jf(x):=c_2\dashint_{\mathbb R^2}\Bigl(\frac{t_j-x_j}{|t-x|^{3}}-\frac{t_j}{(1+|t|^2)^{3/2}}\Bigr)f(t)dt.$$
Here, the normalization constant $c_d$ is given by $c_d:=(\pi \gamma_{d-1})^{-1}$, where $\gamma_{d-1}$ is the Euclidean volume of the $(d-1)$-dimensional Euclidean ball. It is also worth noting that the integrals above converge for almost all $x\in\mathbb R^2.$ 
\end{defin}
\begin{rem}
The space of Lipschitz functions, i.e. functions $f$ satisfying for all $x,y\in \mathbb R^2$ the following inequality: $|f(x)-f(y)|\leq C|x-y|$ with $C>0$ independent of $x,y$, will be denoted by $\mathrm{Lip}(\mathbb R^2)$. By $\mathrm{Lip}(\kappa,\mathbb R^2)$ we shall denote all Lipschitz functions with the Lipschitz constant $C$ satisfying $C\leq \kappa$.  
\end{rem}
\begin{rem}
We accumulate here the list of the frequently used technical abbreviations and notations. For a square (in this article we consider only squares with facets parallel to the coordinate axes) $a\subset \mathbb R^2$ its edge length is denoted by $l(a), c_a$ will stand for the centre of $a$ and $\alpha a$ with $\alpha$ positive will be the square centred at $c_a$ and whose edge length equals $\alpha l(a).$ 
\end{rem}

We are now ready to present the main result of this article.

\begin{theorem} (the two--dimensional Nazarov lemma)
\label{thm1}
Let $\Omega\in L^1(dP)\cap \mathrm{Lip}(\mathbb R^2)$ be a positive function. Then for each $\varepsilon > 0$ there exists a function $\Omega_1
$, satisfying 
\begin{enumerate}[label=(\Alph*)]
\item \label{a} $\Omega(x_1,x_2) \leq \Omega_1(x_1,x_2)$ for all $(x_1,x_2)\in \mathbb R^2;$
\item \label{b} $\Omega_1 \in L^1(dP);$
\item \label{c} $\mathcal R_1\Omega_1 \in \mathrm{Lip}(\varepsilon, \mathbb R^2), \mathcal R_2\Omega_1 \in \mathrm{Lip}(\varepsilon, \mathbb R^2).$
\end{enumerate}
\end{theorem}

Note that Theorem~\ref{thm1} is a two--dimensional generalization of Theorem~\ref{thm0}. Of course, a similar result holds if the space $\mathbb R^2$ is replaced by a general Euclidean space. We state and discuss the corresponding result in the Appendix (see Theorem~\ref{thm2}).

\begin{rem}
Our proof of Theorem~\ref{thm1} relies on the one dimensional pattern of~\cite{nazhav}. However, there is a subtlety: the construction of the regularized system of intervals from~\cite{nazhav} which is the crucial step of the proof of Theorem~\ref{thm0} should be nontrivially adapted for the case of several dimensions (see Remark~\ref{rem0} and Theorem~\ref{thm2}). Moreover, we show that the regularization is a quite general operation, since it can be successfully applied to an arbitrary family of mutually disjoint dyadic squares, see Remark~\ref{mda}. On top of that, as we shall see, most of the estimates from the proof of Theorem~\ref{thm0} are harder in $\mathbb R^n$.
\end{rem}

We hope that our lapidary style will not disturb the reader. In fact, one of the goals of this article was to write down a more ``concrete'' and self-contained version of the proof of Theorem~\ref{thm0}. 

Several comments are in order. As we have already mentioned, the principal step of our proof of Theorem~\ref{thm1} is a suitable generalisation of the regularized family of intervals, constructed in~\cite{nazhav} to several dimensions. Note that this one-dimensional regularized system of intervals is of its own interest and has already obtained applications in harmonic analysis. For instance,  in~\cite{sz} P. Zatitskiy and D. Stolyarov use an appropriate variant of this construction and apply it to some questions related to the Bellman function method.

Alas, our Theorem~\ref{thm1} does not imply the multidimensional Beurling--Malliavin theorem. However, it is possible to deduce from Theorem~\ref{thm1} that every function that satisfies the conditions of the Beurling--Malliavin theorem can be minorized by the absolute value of a function with the spectrum in an arbitrary thin strip. Unfortunately, our main construction from Theorem~\ref{thm1} does not work if the Riesz transforms are replaced by the double Hilbert transform. The corresponding questions alluded in this paragraph seem to the author both hard and intriguing.

The author had asked V.P. Havin in 2014 whether Theorem~\ref{thm1} holds true. Professor Havin suggested that Theorem~\ref{thm1} can be derived directly from Theorem~\ref{thm0} using the rotation method of A. Calder\'on and A. Zygmund. However, I am still not able to deduce Theorem~\ref{thm1} from Theorem~\ref{thm0} directly, i.e. without modifying the proof of the latter theorem.  

\begin{ack}
The author is kindly grateful to Sergei V. Kislyakov for a number of helpful suggestions.
\end{ack}

\section{The proof of the two--dimensional Nazarov lemma}

So, let us prove Theorem~\ref{thm1}.

\begin{proof}

We shall first prove the following \textbf{local} version of the two--dimensional Nazarov lemma.

\begin{lemma}(the local Nazarov lemma)
\label{local}
Let $Q\subset \mathbb R^2$ be a square and let $\delta>0$ and $\kappa>0$ be real numbers. Suppose that $f\in \mathrm{Lip}(\kappa,Q)$ is a nonnegative function such that $\|f\|_{L^\infty(Q)}\leq \delta l(Q).$ Then there exists a function $F\in \mathrm C^{\infty}(\mathbb R^2),$ such that
\begin{enumerate}
\item \label{one}$F=0$ outside $3/2Q,$
\item \label{two} $ f(x_1,x_2)\leq F(x_1,x_2)$ for all $(x_1,x_2)\in Q,$
\item \label{three} $\|\nabla (\mathcal R_1F)\|_{\infty}\lesssim \delta, \|\nabla (\mathcal R_2F)\|_{\infty}\lesssim \delta,$ 
\item \label{four} $\delta^2/\kappa^2\int_{\mathbb R^2}F(x_1,x_2) dx_1 dx_2 \lesssim \int_{Q}f(x_1,x_2)dx_1 dx_2.$
\end{enumerate}
\end{lemma}
\begin{rem}
The signs $\lesssim$ and $\gtrsim$ indicate that the left-hand (right-hand) part of an inequality is less than the right-hand (left-hand) part multiplied by a constant a $C$ independent of $Q$. 
\end{rem}
\begin{proof} 
First, we assume, as we can, that $\delta=1$ and that $Q=[-1/2,1/2]^2=:Q^{*}$. Indeed, we can assume this with no loss of generality since in the general case we can consider the function $\widetilde{f}(x)=f\left(xl(Q)+c_Q\right)/\left(\delta l(Q)\right).$ Note that then $\|\widetilde{f}\|_{L^{\infty}(Q^{*})}\leq l(Q^*)\equiv 1$ and $\widetilde{f}\in \mathrm{Lip}(\kappa/\delta,Q^*).$ Hence there exists a majorant $\widetilde{F}$  of the function $\widetilde{f}$ which equals $0$ outside $3/2Q^*$ and satisfying properties~\ref{three} and~\ref{four} with $\delta=1$ and $\kappa^{\prime}=\kappa/\delta.$ Then the function $F(u):= \delta l(Q) \widetilde{F}\left((u-c_Q)/l(Q)\right)$ will be the needed majorant of the function $f$. Second, from now on we assume, once again without loss of generality, that $\kappa\geq 1$. Indeed, if $\kappa<1$, then it suffices to consider the function $\widetilde{f}(x)=1/\kappa f(x)$. Third, we recall one useful definition, borrowed from~\cite{nazhav}.
\begin{defin}
We call a dyadic square $a$ \textbf{essential}, if $\|f\|_{L^{\infty}(a)}\geq l(a)/2.$ The set of all essential squares will be denoted $A_f.$
\end{defin}
Since the inequality $f(x)>0$ is equivalent to $f(x)>2^{-l_0}$ with some $l_0\in \mathbb N,$
$$\{x\in Q^*:f(x)>0\}=\bigcup_{a\in A_f} a.$$
Next, we define $A_f^m$, the subset of $A_f$ composed of its maximal elements (by inclusion). Note that
$$\bigcup_{a\in A_f} a=\bigcup_{a\in A_f^m} a.$$

We will now associate to each square $a\in A_f^m$ its two dimensional ``tail'' $t(a).$ Informally these tails are defined as follows. The tail $t(a)$ is a family of dyadic squares that is composed of a countable number of finite series $t_p(a),p=1,2,\ldots$ of cells. The squares from the family $t_p(a)$ all have the edge length equal to $l(a)/2^p$ and the union of cells from the series $t_p(a)$ forms the set $\{x\in \mathbb R^2: l(a)/2+\sum_{q=1}^{p-1}([2.5]^q)/2^q l(a)\leq\|x-c_a\|_{\infty}< l(a)/2+\sum_{q=1}^{p}([2.5]^q/2^q) l(a)\}$, where $\|\ldots\|_{\infty}$ stands for the $\sup$ norm in the plane.
See the picture below. \newline

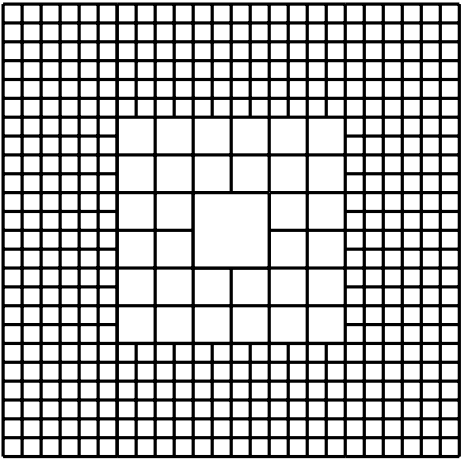
\begin{tikzpicture}
\draw[step=0.5cm, black,very thick] (0,0) grid (3,3);
\fill[white] (1,1) rectangle (2,2);
\draw[step=1cm, black,very thick] (1,1) grid (2,2);
\draw[step=0.25cm, black,very thick] (3,-1.5) grid (4.5,4.5);
\draw[step=0.25cm, black,very thick] (-1.5,-1.5) grid (0,4.5);
\draw[step=0.25cm, black,very thick] (0,-1.5) grid (3,0);
\draw[step=0.25cm, black,very thick] (0,3) grid (3,4.5);
\end{tikzpicture}
\newline
In this picture the biggest square in the center is $a$. The squares of the edge length $l(a)/2$ that surround $a$ form the family $t_1(a).$ The squares of the edge length $l(a)/4$ that constitute $t_2(a).$ 

Here is the formal definition of the tails. If $a=[m2^{-k},(m+1)2^{-k})\times[n2^{-k},(n+1)2^{-k})$ with some $m,k,n\in \mathbb N,$ then $t(a)=\{t_p(a)\}_{p=0}^{+\infty},$ where $t_0(a):=a$ and for $p\geq 1,$ 
\begin{multline*}
t_p(a):=t_{p,\mathrm{up}}(a)\cup t_{p,\mathrm{down}}(a)\cup t_{p,\mathrm{right}}(a)\cup t_{p,\mathrm{left}}(a)\cup \\
t_{p,\mathrm{uprght}}(a)\cup t_{p,\mathrm{uplft}}(a)\cup t_{p,\mathrm{dwnrght}}(a)\cup t_{p,\mathrm{dwnlft}}(a).
\end{multline*}
These eight sets are defined in the following manner. First, we set $\mu_p:=[(2.5)^p],\\ \alpha_1:=2, \beta_1:=-2$ and for $p\geq 2, \alpha_p=2^p+\mu_12^{p-1}+\cdots+\mu_{p-1}2$ and $\beta_p=-\mu_12^{p-1}-\cdots-\mu_{p-1}2-\mu_p.$ Second, define the squares $$Q_{p,i,j}:=[m2^{-k}+j2^{-k-p},m2^{-k}+(j+1)2^{-k-p})\times[n2^{-k}+i2^{-k-p},n2^{-k}+(i+1)2^{-k-p}).$$ The eight sets are now defined as follows
\begin{align*}
t_{p,\mathrm{up}}(a)&=\{Q_{p,i,j}\}_{(i,j)=(\alpha_p,\beta_p+\mu_p)}^{(\alpha_p+\mu_p-1,\alpha_p-1)} & t_{p,\mathrm{down}}(a)&= \{Q_{p,i,j}\}_{(i,j)=(\beta_p,\beta_p+\mu_p)}^{(\beta_p+\mu_p-1,\alpha_p-1)}\\
t_{p,\mathrm{left}}(a)&=\{Q_{p,i,j}\}_{(i,j)=(\beta_p+\mu_p,\beta_p)}^{(\alpha_p-1,\beta_p+\mu_p-1)} & t_{p,\mathrm{right}}(a)&=\{Q_{p,i,j}\}_{(i,j)=(\beta_p+\mu_p,\alpha_p)}^{(\alpha_p-1,\alpha_p+\mu_p-1)} \\
t_{p,\mathrm{uprght}}(a)&=\{Q_{p,i,j}\}_{(i,j)=(\alpha_p,\alpha_p)}^{(\alpha_p+\mu_p-1,\alpha_p+\mu_p-1)} & t_{p,\mathrm{uplft}}(a)&=\{Q_{p,i,j}\}_{(i,j)=(\alpha_p,\beta_p)}^{(\alpha_p+\mu_p-1,\beta_p+\mu_p-1)}\\
t_{p,\mathrm{dwnrght}}(a)&=\{Q_{p,i,j}\}_{(i,j)=(\beta_p,\alpha_p)}^{(\beta_p+\mu_p-1,\alpha_p+\mu_p-1)} & t_{p,\mathrm{dwnlft}}(a)&=\{Q_{p,i,j}\}_{(i,j)=(\beta_p,\beta_p)}^{(\beta_p+\mu_p-1,\beta_p+\mu_p-1)}.
\end{align*}
Note that $\#t_p(a)\lesssim \mu_p^2.$ To understand this, consider the area of the bigger square centred at $c_a$ of edge length $l(a)+2\sum_{q=1}^{p}([2.5]^q/2^q) l(a)$, that consists of about $\#t_p(a)$ smaller squares of edge length $2^{-k-p}.$

Next, we are going to define one system of squares that we shall call $\tau.$ First, we pose $B_f:=\{t(a)\}_{a\in A_f^m}$ and afterwords we define $\tau:=\{c\in B_f^m: c\subseteq Q^*\}.$ We shall prove one important property of the system $\tau.$ To this end, we recall one notation. For a square $a\in \tau$ its neighborhood $N(a)$ is defined as
$$(N_{\tau}(a)=)N(a)=\{b\in \tau: d(a,b)\leq 2l(a),\frac{1}{2}\leq \frac{l(a)}{l(b)}\leq 2\},$$
where $d$ stands for the Hausdorff distance between sets. Note that $\#N(a)\lesssim 1.$ Indeed, this follows from the fact that if $l(a)=2^{-k_0}$ with some $k_0\in \mathbb N,$ then $l(b)$ for $b\in N(a)$ can take only three values: $2^{k_0-1},2^{k_0},2^{k_0+1}.$ 
In the following lemma we establish a property of the system $\tau$ that we shall need.

\begin{lemma}
 Suppose that $a\in \tau$ and $b\in \tau \backslash N(a).$ It follows that if $l(b)\leq 2l(a)$ then $d(2a,2b)\geq l(a)/2$, and if $l(b)=2^k l(a)$ for some natural $k\geq 2$, then $d(2a,2b)\geq 2 \mu_{k-2}l(a).$ 
\end{lemma}
\begin{proof}
We shall first prove the following claim. If $a,b\in \tau$ and $l(a)=2^kl(b)$ for some $k\geq 2, k\in \mathbb N$, then $d(a,b)\geq l(a)\mu_{k-1}/2^{k-1}.$ Indeed, if $b\in t(a),$ then it is obvious from the definition of $t(a).$ Hence, $a,b\in A_f^m.$ Note that if $a\in t(c)$ for some $c\in \tau,$ then $b$ is covered by the squares from $t(c).$ (In fact, the whole $\mathbb R^2$ is covered by them.) We deduce that $$b=\bigcup_{j\in J}b_j$$ for some $b_j\in t(c).$ Indeed, all the squares from the last line are dyadic, but the inclusion $b_j\subset b$ can never happen since $b$ is maximal. Let $b^\star$ be the closest square from $\{b_j\}_{j\in J}$ to the square $a$. Then $d(a,b)=d(a,b^\star)$ and at the same time $a,b^\star\in t(c)$. By the already proved, we infer the inequality $d(a,b)\geq l(a) \mu_{k-1}/2^{k-1},$ and our claim is proved.

Let us now finish the proof of the lemma. If $b\in \tau \backslash N(a), l(b)\leq 2l(a)$ and $l(a)\leq 2l(b)$, then $d(a,b) \geq 2l(a)$ and thanks to the triangle inequality, we infer the following chain of inequalities
$$d(2a,2b)\geq d(a,b)-d(a,2a)-d(b,2b)\geq d(a,b)-\frac{l(a)}{2}-\frac{l(b)}{2}\geq 
\frac{l(a)}{2}.$$ 
On the other hand, if $b\in \tau \backslash N(a), l(b)\leq 2l(a)$ and $l(a)>2l(b),$ then $l(a)\geq 4l(b)$ and hence
$$d(2a,2b)\geq d(a,b)-\frac{l(a)}{2}-\frac{l(b)}{2}\geq \left(\frac{\mu_1}{2}-1\right)l(a)=\frac{l(a)}{2}.$$ 
It remains to rule out the case when $l(b)>2l(a).$ In this case, $l(b)=2^kl(a)$ with some natural $k\geq 2$, and hence
$$d(2a,2b)\geq d(a,b)-l(b)\geq \frac{\mu_{k-1}}{2^{k-1}}l(b)-l(b)\geq 2\mu_{k-2}l(a),$$
and the needed property is proved. 
\end{proof}

The proof of the following remark is now an easy exercise.  
\begin{rem}
\label{mda}
Let $A$ be an arbitrary system of mutually disjoint dyadic squares. Consider the regularized family $B^m$ where $B=\{t(a)\}_{a\in A}$. Suppose that $a\in B^m$ and $b\in \tau \backslash N_{B^m}(a).$ It follows that if $l(b)\leq 2l(a)$ then $d(2a,2b)\geq l(a)/2$, and if $l(b)=2^k l(a)$ for some natural $k\geq 2$, then $d(2a,2b)\geq 2 \mu_{k-2}l(a).$ 
\end{rem}

As a direct consequence of the lemma, we deduce that the multiplicity
$\#\{\alpha\in 2\tau:x\in \alpha\}$ is uniformly bounded in  $x\in \mathbb R^2$. 
Indeed, if $b\in\tau\backslash N(a),$ then $d(2a,2b)>0$ and
$$\sup\limits_{x\in \mathbb R^2}\#\{\alpha\in 2\tau:x\in \alpha\}\leq \sup\limits_{a\in\tau}\#N(a)\lesssim 1.$$

We are ready now to define the function $F$. We first fix a ``cup'' function $\phi$, i.e. $\phi \in \mathrm C^{\infty}(\mathbb R^2)$ satisfying $0\leq \phi(x)\leq 1$ for all $x\in \mathbb R^2$, $\phi\equiv 0$ outside $3/2Q^*$ and $\phi\equiv 1$ on $Q^*$. Second, for a square $a\in \tau$ we pose $$\phi_a(x)=l(a)\phi\Bigl(\frac{x-c_a}{l(a)}\Bigr).$$
Simple calculation shows that $\mathcal R_1\phi_b(x)=l(b)\mathcal R_1\phi\left((x-c_b)/l(b)\right).$ Hence, $$\|\nabla(\mathcal R_1\phi_b)\|_{\infty}\lesssim 1.$$ We finally define $F$ as the following sum $F=\sum_{a\in \tau} \phi_a.$

Now we have to check the required properties of the majorant $F.$ The first one follows readily from the definition of $F.$ To prove the second one, note that for all $a\in \tau$ holds $\|f\|_{L^{\infty}(a)}\leq l(a)$.  Indeed, suppose the contrary, i.e. that $\|f\|_{L^{\infty}(a_0)}> l(a_0)$ for some $a_0\in \tau$. This means that
$$\|f\|_{L^{\infty}(2a_0)}\geq \|f\|_{L^{\infty}(a_0)}> l(a_0)=\frac{l(2a_0)}{2},$$
which in turn signifies that $2a_0$ is an essential square and hence $2a_0\in\tau$. Contradiction with the definition of $\tau.$ From here we deduce that if $x\in a\in \tau$, then $|F(x)|\geq l(a)\geq \|f\|_{L^{\infty}(a)}\geq f(x).$

Next we estimate the integral of the function $F$. To this end, we prove a variant of the Hadamard--Landau inequality which is appropriate for our goals.
\begin{lemma}
\label{eprst}
Let $f$ and $\kappa$ be as above. Suppose that $a\in \tau.$ Then $\int_a f\gtrsim \|f\|^3_{L^{\infty}(a)}/\kappa^2.$
\end{lemma}
\begin{proof}
Let $x^{\prime}\in a$ be a point such that $f(x^{\prime})=\|f\|_{L^{\infty}(a)}.$ Then $f(x)\geq f(x^{\prime})-\kappa|x-x^{\prime}|,$ and hence $\int_a f\geq \mathrm{Vol}(\Gamma_{\psi}),$ where $\psi(x)=\left(f(x^{\prime})-\kappa|x-x^{\prime}|\right)\restr a$ and $\Gamma_{\psi}$ is the subgraph of the function $\psi.$ Denote by $C_{x^{\prime}}$ the set
$$C_{x^{\prime}}=\{(x_1,x_2,x_3)\in\mathbb R^3:0\leq x_3\leq f(x^{\prime}),  \kappa |x^{\prime}-(x_1,x_2)|\leq f(x^{\prime})-x_3 \},$$
(here $|\ldots|$ is the usual Euclidean norm on the plane). Note that $C_{x^{\prime}}$ is nothing but a cut cone centered at the point $(x^{\prime},f(x^{\prime}))$.
Since $\kappa\geq 1\geq \|f\|_{L^{\infty}(a)}/l(a),$ we deduce that
$$\int_a f \geq \mathrm{Vol}(\Gamma_{\psi})\geq \frac{1}{4}\mathrm{Vol}(C_{x^{\prime}}) \gtrsim \|f\|_{L^{\infty}(a)}\Bigl(\frac{\|f\|_{L^{\infty}(a)}}{\kappa}\Bigr)^2=\frac{\|f\|^3_{L^{\infty}(a)}}{\kappa^2},$$
where $\mathrm{Vol}$ denotes the usual Euclidean volume in $\mathbb R^3$. So, the lemma is proved.
\end{proof}

We use this result to estimate the integral of the function $F$:
\begin{multline}
\label{chastnoe}
\int_{\mathbb R^2} F\leq \sum_{b\in A_f^{m}} \int_{\mathbb R^2} \phi_b +  \sum_{c\in A_f^{m}} \sum_{b\in t(c)\backslash c} \int_{\mathbb R^2} \phi_b \leq \sum_{b\in A_f^{m}}l(b)^3 + \sum_{c\in A_f^{m}}\sum_{b\in t(c)\backslash c} l(b)^3 \lesssim \\ 
\sum_{b\in A_f^{m}}\|f\|_{L^{\infty}(a)}^3 + \sum_{c\in A_f^{m}}\sum_{p=1}^{\infty}\sum_{b\in t_p(c)} l(b)^3 \lesssim \sum_{b\in A_f^{m}} \kappa^2\int_b f + \sum_{c\in A_f^{m}}\sum_{p=1}^{\infty} \#t_p(c)\biggl(\frac{l(c)}{2^p}\biggr)^3\lesssim \\
\kappa^2\int_{Q^*}f + \sum_{c\in A_f^{m}} l(c)^3 \lesssim \kappa^2\int_{Q^*} f,
\end{multline}
where the inequality before the last one follows from the fact that $\#t_p(c)\lesssim \mu_p^2\leq (2.5)^{2p}<2^{3p}$.

It is now left to derive the inequalities on the derivatives of the Riesz transformations of the function $F$. We shall only prove the estimate $\|(\mathcal R_1F)^\prime_{I}\|_{L^{\infty}(\mathbb R^2)}\lesssim 1$; the other estimates are left to the reader as an exercise. We shall first obtain this estimate for $x\in \bigcup_{b\in\tau}{2b}.$ Let $a(=a(x))$ be the square from $\tau$ such that $x\in 2a$. We infer the following inequality
\begin{multline*}
|(\mathcal R_1F)^\prime_{I}(x)|\leq \\
\sum_{b\in N(a)}|(\mathcal R_1\phi_b)^\prime_{I}(x)| + \sum_{\substack{b\in \tau\backslash N(a),\\ l(b)\leq 2l(a)}}|(\mathcal R_1\phi_b)^\prime_{I}(x)|+ \sum_{k=2}^{\infty}\sum_{\substack{b\in \tau\backslash N(a),\\ l(b)=2^k l(a)}}|(\mathcal R_1\phi_b)^\prime_{I}(x)|=: S_1+S_2+S_3.
\end{multline*}

We shall estimate the terms $S_1, S_2$ and $S_3$ separately. We begin with the sum $S_1,$ whose estimate turns out to be easy
$$S_1\leq \#N(a) \sup_{b\in \tau}\|(\mathcal R_1\phi_b)^\prime_{I}\|_{L^{\infty}(\mathbb R^2)}\lesssim 1.$$

We further proceed to the second term. We use an easy estimate on the kernel of the Riesz transformation, the discussed property of the system $\tau$ and the fact that the covering $\{2b\}_{b\in\tau}$ has finite multiplicity to write
\begin{multline*}
S_2\lesssim \sum_{\substack{b\in \tau\backslash N(a),\\ l(b)\leq 2l(a)}}\int_{\mathbb R^2}\phi_b(t)\frac{\partial}{\partial x_1}\Bigl(\frac{t_1-x_1}{|t-x|^3}\Bigr)dt \lesssim\sum_{\substack{b\in \tau\backslash N(a),\\ l(b)\leq 2l(a)}}\int_{\mathbb R^2}\frac{\phi_b(t)}{|t-x|^3}dt \lesssim \\
\sum_{\substack{b\in \tau\backslash N(a),\\ l(b)\leq 2l(a)}}\int_b \frac{l(a) dt}{|t-x|^3}\lesssim l(a)\int_{\{|u|\geq l(a)/2\}}\frac{du}{|u|^3}\lesssim 1.
\end{multline*}

The third term can be estimated using the same property of $\tau$:
\begin{multline*}
S_3\lesssim \sum_{k=2}^{\infty}\sum_{\substack{b\in \tau\backslash N(a),\\ l(b)=2^k l(a)}} \int_{2b}\frac{\phi_b(t)dt}{|t-x|^3}\leq \sum_{k=2}^{\infty}2^kl(a)\sum_{\substack{b\in \tau\backslash N(a),\\ l(b)=2^k l(a)}} \int_{2b}\frac{dt}{|t-x|^3} \lesssim \\ \sum_{k=2}^{\infty}2^kl(a) \int_{\{|u|\geq 2\mu_{k-2}l(a)\}}\frac{du}{|u|^3}\lesssim 1,
\end{multline*}
and the lemma for $x\in \bigcup_{b\in\tau}{2b}$ follows. 

Next, if a point $z\in \mathbb R^2$ is situated at a positive distance from the set $\bigcup_{b\in\tau}{2b},$ then denote $x$ to be the closest point to $z$ of this set, and let $a(=a(x))$ be a square as above. We infer the following estimates
\begin{multline*}
|(\mathcal R_1F)^\prime_{I}(z)|\leq \sum_{b\in \tau\backslash N(a)} |(\mathcal R_1\phi_b)^\prime_{I}(z)| + \sum_{b\in N(a)} |(\mathcal R_1\phi_b)^\prime_{I}(z)|\lesssim\\
\sum_{b\in \tau\backslash N(a)} \int_{\mathbb R^2}\frac{\phi_b(t)}{|t-x|^3}dt+\#N(a) \sup_{b\in \tau}\|(\mathcal R_1\phi_b)^\prime_{I}\|_{L^{\infty}(\mathbb R^2)}.
\end{multline*}
Referring to the estimates of the terms $S_1,S_2$ and $S_3$ above, we conclude that the lemma is proved in full generality.
\end{proof}
\begin{rem}
\label{rem0}
The choice of the sequence $\mu_p$ is not the only possible. For instance, one can take sequence $\mu_p=[\lambda]^p$ with any $\lambda\in (2,2^{3/2})$, see the second line of the formula~\eqref{chastnoe}. See also the discussion in the Appendix.
\end{rem}

Let us prove the theorem. Denote $\mu$ to be the Lipschitz constant of the function $\Omega.$ Note that we may assume in the theorem that $\Omega(x)=0$ for $|x|\leq \max(\Omega(0)/\mu,1)=:\sigma.$ Indeed, if it is not the case consider the function $\widetilde{\Omega}(\cdot)=\max(0,\Omega-M)(\cdot),$ where $M:=\max_{x\in B(0,\sigma)} \Omega(x)$. If $\widetilde{\Omega_1}$ is a majorant of $\widetilde{\Omega},$ then $\Omega_1:=\widetilde{\Omega_1}+M$ will be a majorant of $\Omega.$ As a consequence we deduce that $\Omega(x)\leq 2\mu|x|.$ Indeed, this is now obvious if $|x|\leq \sigma,$ and otherwise $\Omega(x)\leq\Omega(0)+\mu|x|\leq 2\mu|x|.$ 

We shall next prove the following inequality 
\begin{equation}
\label{kryakrya}
\Omega(x)\lesssim |x|\biggl(\int_{\{t:|t|\geq |x|\}}\Omega(t)dP(t)\biggr)^{1/3}.
\end{equation}
Indeed, this is obvious once $|x|\leq \sigma,$ and for $|x|>\sigma$ we first write
\begin{equation}
\label{kuku}
\int_{\{t:|t|\geq |x|\}}\Omega(t)dP(t) \geq \int_{\{t:3|x|\geq|t|\geq |x|\}}\Omega(t)dP(t) \gtrsim \int_{\{t:3|x|\geq|t|\geq |x|\}}\Omega(t)\frac{dt}{|x|^3}.
\end{equation}
Since $\Omega$ is Lipschitz, arguing as in Lemma~\ref{eprst} (i.e. comparing the integral of the function $\Omega$ with the volume of a cut cone)  we infer the inequality
\begin{equation}
\label{kukuku}
\int_{\{t:3|x|\geq|t|\geq |x|\}}\Omega(t)dt  \gtrsim \Omega(x)^3.
\end{equation}
Estimates~\eqref{kuku} and~\eqref{kukuku} prove~\eqref{kryakrya}. As a consequence, we may assume that $\Omega(x)\leq \varepsilon |x|.$ Indeed, since $\Omega\in L^1(dP),$ we first find for the given $\varepsilon>0$ a real $R>0$ big enough to guarantee $\int_{\{t:|t|\geq R\}}\Omega(t)dP(t)\leq \varepsilon^3,$ and further modify the function $\Omega$ for $|x|\leq R$ as in the previous paragraph, if needed.

We consider the following decomposition of the plane
$$\mathbb R^2= C_{0,0,0}\cup\bigcup_{k=0}^{\infty}\bigcup_{\substack{(i,j)\in\{-1,0,1\}^2,\\ (i,j)\neq (0,0)}}C_{i,j,k},$$
where $C_{0,0,0}=[-1,1)^2$ and $C_{i,j,k}=[i2^{k},(i+1)2^{k})\times[j2^{k},(j+1)2^{k}).$ We apply Lemma~\ref{local} to the squares $Q:=C_{i,j,k}$ and the corresponding functions $f:=\Omega\restr {C_{i,j,k}}$ for each triplet $(i,j,k)$ as above. Note that then, referring to the inequality $\Omega(x)\leq \varepsilon |x|$ we conclude that $\|f\|_{L^{\infty}(Q)}\leq \varepsilon l(Q)/2$ and hence we can choose $\delta=\varepsilon/2$ in Lemma~\ref{local}. This yields functions $F_{i,j,k}$. We are ready to define the majorant $\Omega_1:$
$$\Omega_1:=\sum_{k=0}^{\infty}\sum_{(i,j)\in\{-1,0,1\}^2}F_{i,j,k}.$$

We shall now check the required properties of $\Omega_1.$ The property~\ref{a} follows obviously from Lemma~\ref{local}. We proceed to~\ref{b}
\begin{multline*}
\int_{\mathbb R^2} \Omega_1(t) dP(t) = \sum_{i,j,k} \int_{\mathbb R^2} F_{i,j,k}(t) dP(t) \lesssim \\ 
\sum_{i,j,k}\int_{3/2 C_{i,j,k}} F_{i,j,k}(t) \frac{dt}{2^{3k}} \lesssim \sum_{i,j,k}\frac{1}{\varepsilon^2}\int_{3/2 C_{i,j,k}} \Omega(t) \frac{dt}{2^{3k}} \lesssim \frac{1}{\varepsilon^2}\int_{\mathbb R^2}\Omega(t) dP(t) <\infty.
\end{multline*}

So, it is left to check that the third conclusion holds. First, fix a point $x\in \mathbb R^2.$ Second, denote by $S(x)$ the square from the family $\mathcal F=\{C_{i,j,k}\}_{\{-1,0,1\}^2\times \mathbb N}$ such that $x\in S(x).$ Next, denote by $V(x)$ the subfamily of $\mathcal F$ consisting of the neighbour squares of $S(x)$ and by $W(x)$ its completion: $W(x)= \mathcal F \backslash V(x).$ Finally, write the function $\Omega_1$ as a sum of two functions as follows:
$$\Omega_1=\sum_{(i,j,k)\in W(x)} F_{i,j,k}+ \sum_{(i,j,k)\in V(x)} F_{i,j,k}=:\omega_1+\omega_2.$$
Since there is only a finite number of squares in the family $V(x)$, we conclude that
$$|(\mathcal R_1\omega_2)^{\prime}_{I}(x)|\lesssim \#V(x) \sup_{(i,j,k)\in V(x)}\|\nabla(\mathcal R_1F_{i,j,k})\|_{\infty}\lesssim \varepsilon,$$
where we have just used~\ref{three} in the last estimate. On the other hand, since $\mathrm{spt} (\omega_1)\subseteq\bigcup_{(i,j,k)\in W(x)}3/2C_{i,j,k}$ we deduce that $$\mathrm{spt} (\omega_1)\subseteq\{t\in \mathbb R^2: |t-x|\geq \frac{l(S(x))}{4}\}\subseteq\{t\in \mathbb R^2: |t-x|\geq \frac{|x|}{16}\}.$$ Hence we arrive at the following chain of inequalities
\begin{multline*}
|(\mathcal R_1\omega_1)^{\prime}_{I}(x)|=\Bigl|\Bigl(\int_{\mathbb R^2}\omega_1(t)\frac{t_1-x_1}{|t-x|^3}dt\Bigr)^{\prime}_{I}\Bigr|=\Bigl|\int_{\mathbb R^2}\omega_1(t)\frac{\partial}{\partial x_1}\Bigl(\frac{t_1-x_1}{|t-x|^3}\Bigr)dt\Bigl|\lesssim
\\ \int_{\mathbb R^2}\omega_1(t)\frac{1}{|t-x|^3}dt \lesssim \int_{\mathbb R^2}\Omega_1(t) dP(t) \lesssim  \frac{1}{\varepsilon^2}\int_{\mathbb R^2}\Omega(t) dP(t) \lesssim \varepsilon,
\end{multline*} 
and the Theorem is proved.

\end{proof}

\section{Appendix}
Here we state and discuss the multidimensional version of Theorem~\ref{thm1}.
\begin{prop} (the multidimensional Nazarov lemma)
\label{thm2}
Let $\Omega\in L^1(dP_n)\cap \mathrm{Lip}(\mathbb R^n)$ be a positive function. Then for each $\varepsilon > 0$ there exists a function $\Omega_1
$, satisfying 
\begin{enumerate}
\item  $\Omega(x) \leq \Omega_1(x)$ for all $x\in \mathbb R^n;$
\item  $\Omega_1 \in L^1(dP_n);$
\item  $R_1\Omega_1 \in \mathrm{Lip}(\varepsilon, \mathbb R^n),\ldots, R_n\Omega_1 \in \mathrm{Lip}(\varepsilon, \mathbb R^n),$ where $R_1, \ldots, R_n$ are the Riesz transformations in $\mathbb R^n.$
\end{enumerate}
\end{prop}
\begin{proof}
The proof is almost identical to the one of Theorem~\ref{thm1} so we leave it to the reader as an exercise. We only remark that the parameters of the main construction will now depend on the dimension $n$. In more details, in this case one can choose the corresponding sequence $\mu_p$ as follows: $\mu_p=[\lambda]^p$ with any $\lambda$ satisfying $2<\lambda<2^{(n+1)/n}.$

\end{proof}
\renewcommand{\refname}{References}


\begin{thebibliography}{99}

\bibitem{belbarul}
A. Baranov, Yu. Belov and A. Ulanovskii, \emph{Gap Problem for Separated Sequences and Beurling--Malliavin Theorem}, J. Fourier Anal. Appl., vol. 23, no. 4, (2017), 877--885.

\bibitem{bermal1}
A. Beurling and P. Malliavin, \emph{On Fourier transforms of measures with compact support}, Acta Math., vol. 107, (1962), 291--309.

\bibitem{bermal2}
A. Beurling and P. Malliavin, \emph{On the closure of characters and the zeros of entire functions}, Acta Math., vol. 118, (1967), 79--93.

\bibitem{bourdyat} J. Bourgain and S. Dyatlov, \emph{Spectral gaps without the pressure condition}, Annals of Math., vol. 187, (2018), 825--867.

\bibitem{makpol}
N. Makarov and A. Poltoratski, \emph{Beurling-Malliavin theory for Toeplitz kernels}, Invent. Math., vol. 180, no. 3, (2010), 443--480.

\bibitem{nazhav}
D. Mashregi, F. Nazarov, and V. Khavin, \emph{The Beurling--Malliavin multiplier theorem: The seventh proof}, Algebra i Analiz, vol. 17, no. 5 (2005), 3--68. 

\bibitem{polt}
A. Poltoratski, \emph{Spectral gaps for sets and measures}, Acta Math., vol. 208, no. 1, (2012), 151--209.

\bibitem{sz}
D. Stolyarov and P. Zatitskiy, \emph{Sharp transference principle for BMO functions and $A^p$ weights}, https://arxiv.org/abs/1908.09497

\end{thebibliography}
\end{document}